\documentclass[12pt,reqno]{amsart}
\usepackage{amsmath,amssymb,amsthm,latexsym}
\usepackage{palatino}
\usepackage{bm,bbm}
\usepackage[all,cmtip]{xy}
\usepackage{url}
\usepackage[colorlinks=true,linkcolor=blue,anchorcolor=blue,citecolor=blue,filecolor=blue,urlcolor=blue]{hyperref}

\newcommand{\A}{\ensuremath{\mathbb{A}}}

\newcommand{\N}{\ensuremath{\mathbb{N}}}

\newcommand{\CC}{\ensuremath{\mathcal{C}}}
\newcommand{\DD}{\ensuremath{\mathcal{D}}}
\newcommand{\EE}{\ensuremath{\mathcal{E}}}

\newcommand{\op}[1]{\ensuremath{{#1}^{\mathsf{op}}}}
\newcommand{\pshx}[1]{\ensuremath{\mathsf{Set}^{\op{#1}}}}
\newcommand{\psh}[1]{\ensuremath{[\op{#1},\mathsf{Set}]}}

\newcommand{\Set}{\ensuremath{\mathsf{Set}}}

\newcommand{\Cat}{\ensuremath{\mathsf{Cat}}}

\newcommand{\Hom}{\ensuremath{\mathsf{Hom}}}

\newcommand{\y}{\ensuremath{\mathsf{y}}} 

\newcommand{\tp}{\ensuremath{\mathsf{tp}}}
\newcommand{\Tm}{\ensuremath{\mathsf{Tm}}}
\newcommand{\Ty}{\ensuremath{\mathsf{Ty}}}

\newcommand{\alg}[1]{\ensuremath{\mathsf{#1}}}

\newcommand{\hook}{\ensuremath{\hookrightarrow}}
\newcommand{\mono}{\ensuremath{\rightarrowtail}}

\newcommand{\epi}{\ensuremath{\twoheadrightarrow}}
\renewcommand{\to}{\ensuremath{\rightarrow}}
\newcommand{\too}{\ensuremath{\longrightarrow}}
\newcommand{\tto}{\ensuremath{\rightrightarrows}}

\newcommand{\G}{\ensuremath{\Gamma}}

\newcommand{\type}{\mathsf{type}}       
\newcommand{\types}[2]{#1 \vdash #2:\type}

\newcommand{\terms}[2]{#1 \vdash #2}
\newcommand{\Gterms}[1]{\terms{\Gamma}{#1}}
\newcommand{\ext}[2]{{#1,#2}}

\newcommand{\Id}{\mathsf{Id}}

\newcommand{\refl}{\mathsf{refl}}

\newcommand{\U}{\ensuremath{\mathcal{U}}}
\newcommand{\UU}{\ensuremath{\,\dot{\mathcal{U}}}}
\renewcommand{\u}{\ensuremath{\mathsf{u}}}
\renewcommand{\t}{\ensuremath{\mathsf{u}}}
\newcommand{\T}{\ensuremath{\mathsf{U}}}
\newcommand{\TT}{\ensuremath{\dot{\mathsf{U}}}}
\newcommand{\n}{\ensuremath{\mathsf{n}}}
\renewcommand{\N}{\ensuremath{\mathsf{N}}}
\newcommand{\NN}{\ensuremath{\dot{\mathsf{N}}}}
\newcommand{\tT}{\ensuremath{{\t}:\TT\to\T}}
\newcommand{\V}{\ensuremath{\mathcal{V}}}
\newcommand{\VV}{\ensuremath{\dot{\mathcal{V}}}}
\newcommand{\SSet}{\ensuremath{\,\dot{\Set}}}

\usepackage{tikz}
\usepackage{pdfpages}
\usepackage{tikz-cd}

\newcommand{\pbcorner}{\arrow[dr,phantom,"\lrcorner" very near start, shift right=.5ex]} 

\newcommand{\pbcornerright}{\arrow[dl,phantom,"\llcorner" very near start, shift left=.5ex]} 

\newtheorem{theorem}{Theorem}
\newtheorem*{theorem*}{Theorem}
\newtheorem{proposition}[theorem]{Proposition} 
\newtheorem{lemma}[theorem]{Lemma}

\theoremstyle{remark}
\newtheorem{remark}[theorem]{Remark} 
\newtheorem*{remarks*}{Remarks}
\newtheorem{example}[theorem]{Example}

\theoremstyle{definition}
\newtheorem{definition}[theorem]{Definition}

\begin{document}

\title{Algebraic Type Theory\\
Part 1: {M}artin-{L}\"of algebras}
\author{Steve Awodey}
\date{\today}
\dedicatory{In memory of Phil Scott, mentor and friend.}

\begin{abstract}
A new algebraic treatment of dependent type theory is proposed using ideas derived from topos theory and algebraic set theory.
\end{abstract}
\maketitle

One of the most beautiful aspects of the book \emph{Introduction to Higher-Order Categorical Logic} by Lambek and Scott is the almost entirely algebraic treatment of higher-order logic, with operations and equations in place of the traditional presentation of logic by rules of inference.   This is made possible in a general way by F.W.\ Lawvere's profound analysis of all of the logical primitives as adjoints \cite{Lawvere:adjointness}, but more specifically by the presence in a topos $\EE$ of a subobject classifier $\Omega$ that represents the presheaf of subobjects via a natural isomorphism of sets,
\[
\mathsf{Sub}(X) \cong \Hom_{\EE}(X, \Omega)\,.
\]
The logical operations on subobjects $\{x\,|\, \varphi(x)\} \mono X$, represented by formulas $\varphi(x) : X \to \Omega$, are themselves represented by operations on $\Omega$, such as conjunction $\wedge : \Omega \times \Omega \to \Omega$.  Thus given two ``propositional functions'' $\varphi(x), \psi(x) : X \to \Omega$ we obtain the meet of their subobjects $\{x\,|\, \varphi(x)\}\cap \{x\,|\, \psi(x)\}$ from the conjunction $\varphi(x)\wedge\psi(x)$ as 
\[
\{x\,|\, \varphi(x)\}\cap \{x\,|\, \psi(x)\} = \{x\,|\, \varphi(x) \wedge \psi(x)\} \,.
\]
The conjunction arises simply by (pairing and) composing:
\begin{equation*}
\begin{tikzcd}
 X \ar[rr, "{\langle \varphi(x), \psi(x) \rangle}" ] \ar[rrd,swap,  "\varphi(x)\wedge\psi(x)"] && \Omega \times \Omega \ar[d, "\wedge"] \\
 && \Omega 
\end{tikzcd}
\end{equation*}
It follows immediately that the operation $\varphi(x)[t(y)/x] = \varphi(t(y))$ of substitution of a term $t(y) : Y\to X$ for the variable $x$ necessarily respects conjunction, just by the associativity of composition:
\begin{align*}
(\varphi(x)\wedge\psi(x))[t(y)/x] &= (\wedge \circ \langle \varphi(x), \psi(x) \rangle) \circ t(y)\\
&= \wedge \circ (\langle \varphi(x), \psi(x) \rangle \circ t(y))\\
&= \wedge \circ \langle \varphi(x)\circ t(y), \psi(x)\circ t(y) \rangle\\
&=  \wedge \circ  \langle \varphi(t(y)), \psi(t(y))\rangle \\
&=  \varphi(t(y))\wedge\psi(t(y)) 
\end{align*}
It therefore also follows that the corresponding meet operation $\cap$ on subobjects also respects substitution, which is interpreted by pullback of subobjects, without invoking the fact that pullback preserves meets.

The same thing holds for the other propositional operations $\top$, $\bot$, $\neg$, $\vee$, $\Rightarrow$, $\Leftrightarrow$, which are also representable.  Moreover, representability also implies the Beck-Chevalley condition for the quantifiers $\forall$ and $\exists$; e.g.\
\begin{align*}
(\forall z .\vartheta(x,z))[t(y)/x] &= (\forall z .\vartheta(x,z))\circ t(y) \\
&= (\forall z \circ \vartheta(x,\hat{z}))\circ t(y)\\
&= \forall z \circ (\vartheta(x,\hat{z})\circ t(y))\\
&= \forall z \circ \vartheta(t(y),\hat{z})\\
&= \forall z .\vartheta(t(y),z)\,,
\end{align*}
in virtue of the universal quantifier $\forall z$ also being represented by a map on $\Omega$, namely $\forall_z : \Omega^{Z} \to \Omega$, with which we simply compose (this time after an exponential transposition).
\begin{equation}\label{diag:BCforQ}
\begin{tikzcd}
 X \ar[rr, "{\vartheta(x,\hat{z})}" ] \ar[rrd,swap,  "{\forall z. \vartheta(x,z)}"] && \Omega^{Z} \ar[d, "{\forall z}"] \\
 && \Omega
\end{tikzcd}
\end{equation}
Again, the corresponding equations then also hold for the classified subobjects.

The general idea is that, because they are natural in the context of variables~$X$, the logical operations on subobjects are represented by ``homming in'' to an algebra of propositions $\Omega$ (by Yoneda, of course). And since they are then just pointwise operations on propositional functions $\varphi(x) : X \to \Omega$, they automatically respect substitutions of terms $t(y): Y \to X$  into the context of variables.  In this way,  the internal logic of a topos arises almost entirely from homming into the internal \emph{complete Heyting algebra} $\Omega$ --- combined with some basic $\lambda$-calculus, enabling higher types constructed from $\Omega$.  This, in a nutshell, is what permits the lovely algebraic formulation of even higher-order logic in \cite{LS:1988}.  

\subsection*{Martin-L\"of algebras.} One of the motivations for the present work was to apply this same approach to \emph{dependent type theory} in place of predicate logic, by determining a suitable algebraic gadget $\U$ in place of $\Omega$, representing the \emph{presheaf of types}, rather than the presheaf of subobjects.  In fibrational terms, over the category $\CC$ of contexts of variables, we would like a representing object $\UU \to \U$ for the \emph{codomain fibration} $\CC^\downarrow \epi \CC$, rather than the object $1\mono \Omega$ representing the fibration of subobjects.  Unlike the discrete fibration, or presheaf $\mathsf{Sub} : \op{\CC} \to \Set$, of subobjects, however, which is (at best) poset valued, the pseudofunctor of slice categories $\CC/_{\!(-)} : \op{\CC} \to \mathsf{Cat}$ cannot be representable, even in the weaker sense of a natural equivalence of indexed categories,
\[
\CC/_{X}\, \simeq\, \Hom_{\CC}(X, \U)\,.
\]

There are really two different problems here: size and coherence of structure.  We solve both simultaneously by taking $\UU \to \U$ to determine a full internal subcategory (with suitable additional structure) in the category $\EE = \widehat{\CC}$ of presheaves over $\CC$, splitting the codomain fibration as in \cite{Lums-Warren}.  Whatever one may think of this solution, it is of obvious interest to determine what additional algebraic structure on $\UU \to \U$ will serve to model dependent type theory. And we know how to find the answer: by Yoneda!  We call the resulting gadget a \emph{Martin-L\"of algebra} (``ML-algebra'' for short), and it turns out to have a remarkably simple description, as an algebra for a polynomial monad, giving a complete answer to our question (Theorem~\ref{thm:nmcwf}).

The polynomial endofunctor in question $\alg{P}_\t : \EE \to \EE$  is that of the algebra $\u:\UU \to \U$, which therefore has the form
\[
\alg{P}_{\u}(X) = \Sigma_{A: \U} X^A\,,
\]
allowing e.g.\ the formation rule for the $\Pi$-type to be expressed by a composition:
\begin{equation*}
\begin{tikzcd}
 \Gamma \ar[rr, "{(A,B)}" ] \ar[rrd,swap,  "{\Pi_A B}"] &&\alg{P}_\u (\U) \ar[d, "{\Pi}"] \\
 && \U
\end{tikzcd}
\end{equation*} 
This ensures not only the strict Beck-Chevalley rules for the type formers, as in \eqref{diag:BCforQ}, but also a solution (due originally to Voevodsky) to the old bugbear of coherence in dependent type theory \cite{Hofmann:1994}, since substitution $\sigma : \Delta \to \Gamma$ into types and terms $\Gamma \vdash a:A$ is now interpreted simply by composition, which unlike pullback, is strictly associative.
\[
\begin{tikzcd}
 	&&&  {\UU} \ar[d]\ar[d]\\
\Delta \ar[r,  "{\sigma}"] \ar[rrru, bend left = 4ex, "{a[\sigma]}"]\ar[rrr, swap, bend right = 5ex, "{A[\sigma]}"] & \Gamma \ar[rru, "a"]   \ar[rru, swap, "{\ \ a:A}"]  \ar[rr ,swap,  "A"]  && {\U}
\end{tikzcd}
\]

The full rules for $\Pi$-types turn out to state exactly that $\u : \UU\to\U$ is an algebra for its own polynomial endofunctor $\alg{P}_{\u} : \widehat{\CC} \to \widehat{\CC}$ when the latter is lifted to the (cartesian) arrow category $\alg{P}_{\u}^{\downarrow}: \widehat{\CC}^{\,\downarrow}_{\mathsf{cart}} \too \widehat{\CC}^{\,\downarrow}_{\mathsf{cart}}$.
Even more strikingly, the type formers $1, \Sigma$ endow $\alg{P}_{\u} : \widehat{\CC} \to \widehat{\CC}$ with the underlying structure of a polynomial monad.  The monad and algebra laws are then seen to express fundamental type isomorphisms, the analysis of which requires the important $\Id{}$-types (as recently reformulated in terms of polynomials by R.\ Garner), which forms one of the main new advances of the current work.

\subsection*{Algebraic type theory} There turns out to be an intriguing analogy between ML-algebras and the Zermelo-Fraenkel algebras from the \emph{Algebraic Set Theory} of Joyal and Moerdijk \cite{JM:AST}.  A second motivation for the present work was to explore this analogy, which is discernible in the similarity between the universal small map $\pi : E \to B$ and an ML-algebra $\u : \UU \to \U$, and between a ZF-algebra structure map $P_sV \to V$ (especially in the formulation of \cite{Simpson:1999}) and the $\Sigma$-type former $\alg{P}_{\!\u}\, \U \to \U$, as well as between the respective monads $P_s$ and $\alg{P}_{\!\u}$. 

Indeed, an ML-algebra is in some sense a proof-relevant version of the ZF-algebras from \emph{op.\ cit.}  To be sure, only the most basic aspects of this connection have been developed here.   Apart from the obviously missing successor operation $s : V\to V$, one still needs to consider morphisms of ML-algebras, free and initial algebras, as well as the functor induced by a change of context $\CC \to \CC'$. Some of this is underway in the work in progress \cite{Fernando}.
\medskip

We begin in Section 1 below by recalling from \cite{awodey:NM} the notion of a \emph{natural model} of dependent type theory and its relation to the \emph{categories with families} of \cite{Dybjer:CWF}.  In section \ref{sec:MLalgebras} we abstract the main features of a natural model with the type formers $\mathsf{1}, \Sigma, \Pi$ to form the notion of a \emph{Martin-L\"of algebra}, the basic theory of which is also begun in this section with the addition of identity types.   After briefly indicating the relation to the \emph{tribes} of Joyal \cite{Joyal:CandT} we provide some examples of ML-algebras in Section \ref{sec:ExamplesML-Algebras}, including the subobject classifier $\Omega$ of a topos, and the Hofmann-Streicher universe $\VV\to\V$ in presheaves.  The rest of the paper is devoted to the relationship between ML-algebras and polynomial monads, which was already considered from a somewhat different point of view in \cite{ Newstead:thesis} and \cite{NA:2018}.  We conclude in Section \ref{sec:Monad} with a new perspective on that issue, as well as on a familiar example. 

\section{Natural models of type theory}\label{sec:NaturalModels}

\noindent We write $\widehat{\CC} = \psh{\CC}$ for the category of presheaves on a small category~$\CC$. 
In \cite{awodey:NM}, a \emph{natural model of type theory} is defined to be a representable natural transformation $\t: \TT\to \T$ of presheaves in $\widehat{\CC}$.  

\begin{definition}\label{def:representablenattrans} A natural transformation $p : Y \to X$ of presheaves on a category $\CC$  is \emph{representable} if the pullback of $p$ along any element $x: \y{C} \to X$ is representable.
\begin{equation*}
\begin{tikzcd}
 \y{D} \ar[d,swap,"{\y{c}}"] \ar[r,"y"] \pbcorner &  Y \ar[d, "p"]\\
\y{C} \ar[r,swap,"x"]  & X
\end{tikzcd}
\end{equation*}
We may assume a choice of pullback data $c : D \to C$ in $\CC$ and $y \in Y(D)$ for all $x\in X(C)$ (but no coherence conditions).
\end{definition}

Proposition 2 of \emph{op.\ cit.} shows that such a map is essentially the same thing as a \emph{category with families} (CwF) in the sense of \cite{Dybjer:CWF} when $\CC$ is regarded as the category of contexts of a type theory, $\T: \op{\CC} \to \Set$ is regarded as the presheaf of types in context, $\TT: \op{\CC} \to \Set$ as the presheaf of terms in context, and $\t : \TT\to\T$ as the typing of the terms.   

\begin{proposition}[\cite{awodey:NM, Fiore:2012}] A representable natural transformation is the same thing as a \emph{category with families (CwF)} in the sense of Dybjer \cite{Dybjer:CWF}.
\end{proposition}

We sketch the correspondence from \cite{awodey:NM}.
Let us write the objects and arrows of \CC\ as $\sigma : \Delta \to \Gamma$, giving the \emph{category of contexts and substitutions}.
A CwF is usually defined as a presheaf of \emph{types in context}, 
\[
\Ty : \CC^{\mathrm{op}}\to \Set\,,
\]
together with a presheaf of \emph{typed terms in context}, 
\[
\Tm : {\textstyle (\int_{\CC}\!\Ty)^{\mathsf{op}}\to \Set }\,.
\]
We reformulate this using the familiar equivalence
\[\textstyle
 \pshx{\CC}\!/_{\Ty}\ \simeq\ \pshx{(\int_{\CC}\Ty)}
\]
in order to obtain a map $\tp :\Tm \to \Ty$.

Formally, we then interpret:
\begin{align*}
 \Ty(\Gamma) &= \{ A \,|\, \Gamma \vdash A\} \\
 \Tm(\Gamma) &= \{ a \,|\, \Gamma \vdash a:A\} 
\end{align*}
Under the Yoneda lemma we therefore have a bijective correspondence:
\begin{align*}
\Gamma\vdash A\quad &\approx\quad A : \y{\Gamma}\to \Ty  \\
\Gterms{a:A}\quad &\approx\quad  a: \y{\Gamma}\to \Tm  \qquad (\tp\circ a = A)
\end{align*}
as indicated in the following.
\[
\begin{tikzcd}
 \strut \ar[d, phantom,"\Gamma \vdash a:A"] && \strut \ar[d, phantom,"\approx"] &&  && {\Tm} \ar[d,"{\tp}"]\\
\strut && \strut && \y\Gamma \ar[rru, "a"]   \ar[rr ,swap,  "A"]   && {\Ty}
\end{tikzcd}
\]

The action of a substitution of contexts $\sigma : \Delta \to \Gamma$ on types and terms,
\[
\frac{\sigma: \Delta \to \Gamma, \quad \Gamma\vdash a:A}{\Delta\vdash a[\sigma] : A[\sigma]}
\]
is then interpreted simply as composition:
\[
\begin{tikzcd}
 	&&&  {\Tm} \ar[d,"{\tp}"]\\
\y{\Delta} \ar[r,  "\y{\sigma}"] \ar[rrru, bend left = 4ex, "{a[\sigma]}"] \ar[rrr, swap, bend right = 5ex, "{A[\sigma]}"] & \y\Gamma \ar[rru, "a"]   \ar[rr ,swap,  "A"]  && {\Ty}
\end{tikzcd}
\]
We may hereafter omit the $\y$ for the Yoneda embedding, letting the Greek letters serve to distinguish representable presheaves and their maps.

The CwF operation of \emph{context extension} 
$$\frac{\quad\Gamma\vdash A\quad}{\ \ \ext{\G}{A}\vdash}$$
is modeled by the representability of $\tp : \Tm\to\Ty$ as follows.
Given $\Gamma\vdash A$ we need a new context $\ext{\G}{A}$ together with a substitution $\pi_A : \ext{\G}{A} \to \G$ and a term 
\[
\terms{\ext{\G}{A}}{q_A:A[\pi_A]}\,.
\]
This models the usual  rule of ``weakening the context'' from $\types{\G}{A}$ to $\terms{\ext{\G}{x:A}}{x:A}$. The substitution $A[\pi_A]$ is given by composition $A[\pi_A] = A\circ \pi_A$, where $\pi_A: \ext{\G}{A} \to \G$ is the pullback of $\tp$ along $A$, which exists as an arrow in $\CC$ since $\tp$ is representable.
\begin{equation}\label{diag:NMpullback}
\begin{tikzcd}
{\ext{\G}{A}} \ar[d,swap, "{\pi_A}"] \ar[r, "{q_A}"] \pbcorner &  {\Tm} \ar[d, "{\tp}"]\\
{\G} \ar[r,swap,"A"]   & {\Ty}
\end{tikzcd}
\end{equation}
The map $q_A : \ext{\G}{A}\to\Tm$ arising from the pullback then gives the required term $\terms{\ext{\G}{A}}{q_A:A[\pi_A]}$ since $\tp\circ q_A = A \circ \pi_A = A[\pi_A]$.
The remaining laws of a CwF follow from the pullback condition on $\eqref{diag:NMpullback}$; see \cite{awodey:NM}.

\subsection{Modeling the type formers}

Given a natural model $\tT$, we will make extensive use of the associated \emph{polynomial endofunctor}
$\alg{P}_\t : \widehat\CC \too \widehat\CC$ (cf.\ \cite{GambinoKoch:2013}), defined by 
\[
\alg{P}_\t = \T_! \circ \t_* \circ \TT^* : \widehat{\CC}\too\widehat{\CC}\,,
\] 
as indicated below.
\begin{equation*}
\xymatrix{
  \Set^{\CC^\mathsf{op}} \ar[d]_{\TT^*} \ar[rr]^{\alg{P}_\t} && \Set^{\CC^\mathsf{op}} \\
  \Set^{\CC^\mathsf{op}}\!/_{\TT} \ar[rr]_{{\t}_*} && \Set^{\CC^\mathsf{op}}\!/_\T \ar[u]_{{\T}_!} 
 }
\end{equation*}

The action of $\alg{P}_\t$ on an object $X$ may be depicted:
\begin{equation*}
\xymatrix{
 X & \ar[l] X\times \TT \ar[d] && \alg{P}_\t(X) \ar[d] \\
	& \TT \ar[rr]_{\t} && \T 
 }\qquad\quad
 \end{equation*}

We call $\tT$ the \emph{signature} of $\alg{P}_\t$ and briefly recall the following \emph{universal mapping property} from \cite{awodey:NM}.

\begin{lemma}
 For $p: E \to B$ in a locally cartesian closed category $\EE$ we have the following universal property of the polynomial functor $\alg{P}_p$\,. For any objects $X, Y \in \EE$, maps $f : Y \to \alg{P}_p (X)$ correspond bijectively to pairs of maps $f_1 : Y \to B$ and $f_2 : Y\times_B E \to X$, as indicated below.
\begin{equation}\label{diag:polyUMP}
\begin{tikzcd}
& Y \ar[rr,"f"] && \alg{P}_p (X) & \\[-5ex]
\strut \ar[rrrr, equal] &&&& \strut\\[-5ex]
&X & \ar[l,"f_2"] Y\times_B E \ar[d] \ar[r] \pbcorner & E \ar[d, "p"] &\\
&& Y \ar[r,swap,"f_1"] & B &
\end{tikzcd}
\end{equation}
The correspondence is natural in both $X$ and $Y$, in the expected sense.
\end{lemma}

This universal property is also suggested by the conventional type theoretic notation, namely: $$\alg{P}_p (X) = \Sigma_{b:B} X^{E_b}$$
The lemma can be used to determine the signature $\, p\cdot q\, $ for the composite $\alg{P}_p \circ \alg{P}_q$ of two polynomial functors, which is again polynomial, and for which we have
\begin{equation}\label{eq:polycomp}
\alg{P}_{p\cdot q}\, \cong\, \alg{P}_p \circ \alg{P}_q\,.
\end{equation}
Indeed, let $p : B \to A$ and $q : D \to C$, and consider the following diagram resulting from applying the correspondence \eqref{diag:polyUMP} to the identity arrow,
\[
\langle a, c\rangle = 1_{\alg{P}_p(C)} : \alg{P}_p(C) \to \alg{P}_p(C)\,,
\]
and taking $Q$ to be the indicated pullback.
\begin{equation}\label{diag:polyUMP2}
\begin{tikzcd}
D \ar[d,swap,"q"] & \ar[l] Q \ar[d] \pbcornerright  \ar[dd, dotted, bend left = 10ex ,near start, "{p\cdot q}"]& & \\
 C  & \ar[l,"c"] \pi^*B \ar[d] \ar[r] \pbcorner & B \ar[d,"p"] \\
&  \alg{P}_p(C) \ar[r,swap,"a"] & A
\end{tikzcd}
\end{equation}
The map $p\cdot q$ is then defined to be the indicated composite,
\[
p\cdot q\ =\ a^*p \circ c^*q\,.
\]
The condition \eqref{eq:polycomp} can then be checked using the correspondence \eqref{diag:polyUMP} (also see \cite{GambinoKoch:2013}).

\begin{definition}\label{def:modelthetypeformers}
A natural model $\tT$ over $\CC$ will be said to \emph{model} the type formers $ \mathsf{1}, \Sigma, \Pi$ if there are pullback squares in $\widehat{\CC}$ of the following form,
\begin{equation}\label{diag:naturalmodelPiSigma1}
\begin{tikzcd}
	1 \ar[r] \ar[d] \pbcorner &  \TT \ar[d, "\t"]\\  
	1 \ar[r] & \T
 \end{tikzcd} \qquad \quad 
 \begin{tikzcd}
	\TT_2  \ar[r] \ar[d,swap,"\t\cdot\t"] \pbcorner &  \TT \ar[d, "\t"]\\  
	\T_2 \ar[r] & \T
	 \end{tikzcd} \qquad \quad 
	 \begin{tikzcd}
	 \alg{P}_\t(\TT) \ar[r] \ar[d,swap,"\alg{P}_\t(\t)"] \pbcorner &  \TT \ar[d, "\t"]\\  
	\alg{P}_\t(\T) \ar[r] & \T
 \end{tikzcd} 
 \end{equation}
where $\t\cdot\t : \dot{\T}_2 \to \T_2$ is determined by $\alg{P}_{\t\cdot\t} \cong  \alg{P}_{\t} \circ \alg{P}_{\t}$ as in \eqref{eq:polycomp}.
\end{definition}
The terminology is justified by the following result from \cite[Theorem 16]{awodey:NM}.
\begin{theorem}\label{thm:nmcwf}
Let $\tT$ be a natural model. The associated category with families satisfies the usual rules for the type-formers $\mathsf{1}, \Sigma, \Pi$ just if $\tT$ models the same in the sense of Definition \ref{def:modelthetypeformers}.
\end{theorem}

We only sketch the case of $\Pi$-types; the other type formers are treated in detail in \cite{awodey:NM, Newstead:thesis, NA:2018}.

\begin{proposition}  The natural model $\tT$ models $\Pi$-types just if there are maps $\lambda$ and $\Pi$ making the following a pullback diagram. 
\begin{equation}\label{diag:prod2}
\begin{tikzcd}
\alg{P}_\t(\TT)  \ar[d,swap, "{\alg{P}_\t(\t)}"]   \ar[r, "\lambda"]  \pbcorner &  \TT \ar[d, "{\t}"]\\
\alg{P}_\t(\T)	\ar[r,swap, "{\Pi}"] & \T
\end{tikzcd}
\end{equation}
\end{proposition}
\begin{proof}
Unpacking the definitions, we have $\alg{P}_\t(\T) = \Sigma_{A:\T} \T^A $, etc.,  
so diagram \eqref{diag:prod2} becomes:
\begin{equation*}
\begin{tikzcd}
\Sigma_{A:\T} \TT^A \ar[d,swap, "{\Sigma_{A:\T} \t^A}"]  \ar[r, "{\lambda}"] &  \TT \ar[d,"{\t}"] \\
\Sigma_{A:\T} \T^A    \ar[r, swap,"{\Pi}"] & \T
\end{tikzcd}
\end{equation*}
For $\G \in \CC$, maps $\G\to \Sigma_{A:\T} \T^A $ correspond to pairs $(A,B)$ with  $A:\G\to \T$ and $B : \ext{\G}{A} \to \T$, and thus to $\G\vdash A$ and $\ext{\G}{A} \vdash B$. Similarly, a map $\G\to \Sigma_{A:\T} \TT^A $ corresponds to a pair $(A,b)$ with $\G\vdash A$ and $\ext{\G}{A} \vdash b : B$, the typing of $b$ resulting from composing with the map 
\[
\Sigma_{A:\T} \t^A : \Sigma_{A:\T} \TT^A \to \Sigma_{A:\T} \T^A\,.
\]
\begin{equation*}
\begin{tikzcd}
	& \Sigma_{A:\T} \TT^A \ar[dd]  \ar[rr, "{\lambda}"] &&  \TT \ar[dd,"{\t}"] \\
\G \ar[dr,swap, "{(A,B)}"] \ar[ru, "{(A,b)}"] \ar[urrr, swap, crossing over, "{ \lambda_{A}b }"] \ar[drrr, crossing over, "{ \Pi_{A}B }"] &&& \\
 	& \Sigma_{A:\T}\T^A    \ar[rr, swap,"{\Pi}"] && \T 
\end{tikzcd}
\end{equation*}
The composition across the top is then the term $\G \vdash \lambda_{x:A} b$ , the type of which is determined by composing with $\t$ and comparing with the composition across the bottom, namely $\G \vdash \Pi_{x:A} B$. In this way, the lower horizontal arrow in the diagram models the \emph{$\Pi$-formation rule}:
\[
\frac{\quad\ext{\G}{A}\vdash B\quad}{\G\vdash \Pi_{x:A} B}
\]
and the upper horizontal arrow, along with the commutativity of the diagram, models the \emph{$\Pi$-introduction rule}:
\[
\frac{\ext{\G}{A}\vdash b:B}{\G\vdash \lambda_{x:A} b : \Pi_{x:A} B}
\]
The square \eqref{diag:prod2} is a pullback just if, for every $(A,B) : \G \to \Sigma_{A:\T}\T^A$ and every $t: \G \to \TT$ with $\t \circ t = \Pi_A B$, there is a unique $(A,b) : \G \to \Sigma_{A:\T}\TT^A$ with $b:B$ and $\lambda_A b = t$.  In terms of the interpretation, given $\ext{\G}{A} \vdash B$ and $\G\vdash t: \Pi_{x:A} B$, there is required to be a term $\ext{\G}{A}\vdash t':B$ with $\lambda_{x:A}  t' = t$, and $t'$ is unique with this property.  This is just what is provided by the \emph{$\Pi$-elimination rule}:
\[
\frac{{\ext{\G}{A} \vdash B}\qquad {\G \vdash t : \Pi_{x:A}  B} \qquad {\G \vdash x:A}}{{\ext{\G}{A}\vdash t\, x :B}}
\]
in conjunction with the \emph{$\Pi$-computation rules}:
\begin{align*}
\lambda_{x:A} (t\,x) &= t : \Pi_A B\\
(\lambda_{x:A} b)\,x &= b : B
\end{align*}
\end{proof}

\section{Martin-L\"of algebras}\label{sec:MLalgebras}

Now let $\EE$ be a locally cartesian closed category (lccc) and $\tT$ a map in~$\EE$.  As in the foregoing case where $\EE$ was a category of presheaves $\EE=\widehat{\CC}$, the map $\t$ gives rise to a polynomial endofunctor,
$$\alg{P}_\t = \T_! \circ \t_* \circ \TT^* : \EE\too\EE\,,$$
which we may use to define the following abstraction of the notion of a natural model.
\begin{definition}\label{def:MLalg}
A \emph{Martin-L\"of algebra} in an lccc $\EE$ is a map $\tT$ equipped with structure maps $(*, 1, \sigma, \Sigma, \lambda, \Pi)$ making pullback squares 
\begin{equation}\label{diag:MLalgDef}
\begin{tikzcd}
	1 \ar[r,"*"] \ar[d,swap,"!"] \pbcorner &  \TT \ar[d, "\t"]\\  
	1 \ar[r, swap,"\mathsf{1}"] & \T
 \end{tikzcd}\qquad \qquad 
 \begin{tikzcd}
	\TT_2  \ar[r,"\sigma"] \ar[d,swap,"\t\cdot\t"] \pbcorner &  \TT \ar[d, "\t"]\\  
	\T_2 \ar[r,swap, "\Sigma"] & \T
\end{tikzcd} \qquad \qquad 
	 \begin{tikzcd}
	\alg{P}_\t\TT \ar[r,"\lambda"] \ar[d,swap,"\alg{P}_\t(\t)"] \pbcorner &  \TT \ar[d, "\t"]\\  
	\alg{P}_\t\T \ar[r,swap, "\Pi"] & \T
 \end{tikzcd} 
 \end{equation}
 where the map $\t\cdot\t$ is defined in terms of $\alg{P}_\t$ as in \eqref{diag:naturalmodelPiSigma1} via
 \[
\alg{P}_{\t\cdot\t} = \alg{P}_{\t}\circ \alg{P}_{\t}\,.
 \]
   \end{definition}

In place of representability in the elementary setting we may sometimes require the further condition that $\tT$ be $\emph{tiny}$ in the following sense.

\begin{definition}\label{def:tinymap}
 A map $p : Y \to X$ in a locally cartesian closed category $\EE$ will be said to be \emph{tiny} if it is so as an object in $\EE/_{\!X}$ in the sense that exponentiation by $p$ has a right adjoint $(-)^p \dashv (-)_p$. 
\end{definition}

Note that a map $p : Y \to X$ in an lccc is tiny just if the \emph{pushforward} functor $p_* : \EE/_Y \to \EE/_X$ has a right adjoint,
\[
f_* \dashv f^! : \EE/_X \too \EE/_Y\,.
\]

\begin{proposition}\label{prop:rep_is_tiny}
If $\CC$ has finite limits, a map $p : Y \to X$ in $\widehat{\CC}$ is representable just if it is tiny in the sense of Definition \ref{def:tinymap}, which is the case just if the \emph{pushforward functor} (the right adjoint to pullback) 
\[
p^* \dashv p_* : \widehat{\CC}/_{Y} \longrightarrow \widehat{\CC}/_{X}
\] 
itself has a \emph{right} adjoint:
\[
p_! \dashv p^*\dashv p_* \dashv  p^!
\]
\end{proposition}

\begin{proof}
The elementary definition \ref{def:representablenattrans} clearly states that, for the category of elements ${\int\!X} \simeq {\y{}/_X}$, the composition functor 
\[\textstyle
{\int\!p} : {\int\!Y} \to {\int\!X}
\]  
has a right adjoint, say ${\int\!p} \dashv ({\int\!p})^\sharp$. Now recall that $\widehat{\CC}/{Y} \simeq \psh{(\int\!X)}$, so that the precomposition functor $({\int\!p})^* = \psh{({\int\!p})}$ gives rise to a commutative diagram with left and right Kan extensions:
\[\textstyle
({\int\!p})_!\dashv ({\int\!p})^* \dashv ({\int\!p})_*
\]
\begin{equation}\begin{tikzcd}
\psh{(\int{\!Y})} \ar[rr,shift right = 1ex,swap, "{\int\!p}_!"]  \ar[rr,shift left = 1ex] &&   \psh{(\int{\!X})} \ar[ll] \\
&& \\
\int{\!Y} \ar[uu, hook, "{\y{}}"]  \ar[rr, swap, "{\int\!p}"]  &&  \int{\!X} \ar[uu, swap,hook, "{\y{}}"]
\end{tikzcd}\end{equation}
But since ${\int\!p} \dashv ({\int\!p})^\sharp$ , there is a further right adjoint $({\int\!p})_* \dashv (({\int\!p})^\sharp)_*$ to precomposition with $({\int\!p})^\sharp$.  Moving back across the equivalence $\psh{(\int\!X)} \simeq \widehat{\CC}/_{X}$ we obtain the claimed further right adjoint: 
\[
p_! \dashv p^* \dashv p_*\dashv p^! :  \widehat{\CC}/_{X} \too \widehat{\CC}/_{Y}
\]
Conversely, a right adjoint $$\textstyle ({\int\!p})_* \dashv R : \psh{(\int{\!X})} \to \psh{(\int{\!Y})}$$ is easily shown to be induced by precomposing with a right adjoint ${\int\!p} \dashv r : {\int\!X} \to {\int\!Y}$ if $\CC$ has all finite products and idempotents split in $\CC$, which is the case if $\CC$ has finite limits.

We leave the construction of the right adjoint $p^!:  \widehat{\CC}/_{X} \too \widehat{\CC}/_{Y}$ from the right adjoint $(-)^p \dashv (-)_p : \widehat{\CC}/_{X} \to \widehat{\CC}/_{X}$ to the reader.
\end{proof}

\begin{remark}\label{remark:clans}
In \cite{awodey:NM} it is shown how a \emph{clan} in the sense of Joyal \cite{Joyal:CandT}, or \emph{category with display maps} in the sense of \cite{Taylor:PFM}, say $(\CC, \DD)$, gives rise to a natural model $\t : \TT\to\T$ in $\widehat\CC$, namely with 
\[
\t\, =\ \coprod_{d\in\DD}\y{d} \,.\]
In particular, the fibrations in any right-proper Quillen model category (which determine a \emph{$\Pi$-tribe} on the fibrant objects, in the language of \cite{Joyal:CandT}), or simply all maps in an lccc, are shown to give rise to a natural model with $1, \Sigma, \Pi$, and thus a Martin-L\"of algebra, which moreover has \emph{identity types}, in the sense of the next section.
\end{remark}

\subsection{Identity Types}\label{sec:Eq and Id}

A natural model $\tT$ in a presheaf category $\widehat{\CC}$ was determined in \cite{awodey:NM} to model either the extensional or intensional identity types of Martin-L\"of type theory in terms of the existence of certain additional structures.  We transfer these definitions to the elementary setting of Martin-L\"of algebras.  Condition (\ref{def:equalitytype}) below is shown in \cite{awodey:NM} to capture the extensional identity types of Martin-L\"of type theory.  The condition given in \emph{op.cit.}\ for the intensional case is replaced in (\ref{def:identitytype}) below by a simplification suggested by R.\ Garner.   

\begin{definition}\label{def:EqualityandIdentityTypes}
Let $\tT$ be a map in an lccc $\EE$. 
\begin{enumerate}
\item\label{def:equalitytype}  $\tT$ is said to  model the (extensional) \emph{equality type former} if there are structure maps $(\refl, \mathsf{Eq})$ making a pullback square:
\begin{equation*}\label{diag:Eqtypes}
\begin{tikzcd}
	\TT \ar[r,"\refl"] \ar[d,swap,"\delta"] \pbcorner &  \TT \ar[d, "\t"]\\  
	\TT \times_\T \TT \ar[r, swap,"\mathsf{Eq}"] & \T
 \end{tikzcd}\qquad \qquad 
 \end{equation*}
\item\label{def:identitytype}   $\tT$ is said to model the (intensional) \emph{identity type former} if there are structure maps $(\mathsf{i}, \mathsf{Id})$ making a commutative square,
\begin{equation}\label{diag:Idtypes}
\begin{tikzcd}
	\TT \ar[r,"\mathsf{i}"] \ar[d,swap,"\delta"] &  \TT \ar[d, "\t"]\\  
	\TT \times_\T \TT \ar[r, swap,"\mathsf{Id}"] & \T
 \end{tikzcd}\qquad \qquad 
 \end{equation}
 together with a weak pullback structure $\mathsf{J}$ for the resulting comparison square, in the sense of \eqref{def:wps} below. 
 \end{enumerate}
   \end{definition}

To describe the map $\mathsf{J}$, let us see how (\ref{def:identitytype}) models identity types.  Under the interpretation already described in Section \ref{sec:NaturalModels} the maps $\Id$ and $\mathsf{i}$  in 
\begin{equation*}
\begin{tikzcd}
	\TT \ar[r,"\mathsf{i}"] \ar[d,swap,"\delta"] &  \TT \ar[d, "\t"]\\  
	\TT \times_\T \TT \ar[r, swap,"\mathsf{Id}"] & \T
 \end{tikzcd}\qquad \qquad 
 \end{equation*}
respectively, directly model the formation and introduction rules.
\begin{align*}
x, y:A &\vdash \Id_A(x, y)\\
x: A &\vdash \mathsf{i}(x) : \Id_A(x, x)
\end{align*}
Next, pull $\t$ back along $\mathsf{Id}$ to get an object $\mathsf{I}$ and a map $\rho  : \TT \to \mathsf{I}$,
\begin{equation*}
\begin{tikzcd}
\TT \ar[rd,swap,"{\delta}"] \ar[r,"{\rho}"] \ar[rr, bend left,"{\mathsf{i}}"] & \mathsf{I} \ar[r]  \ar[d] \pbcorner &  \TT \ar[d, "\t"]\\
& \TT \times_\T \TT  \ar[r, swap,"{\mathsf{Id}}"]  & \T
\end{tikzcd}
\end{equation*}
which commutes with the compositions to $\T$ as indicated below.
\begin{equation*}
\xymatrix{
\TT \ar[rd]_{\t} \ar[r]^{\rho} & \mathsf{I} \ar[d]^{q}  \\
& \T 
}
\end{equation*}
The map $\rho  : \TT \to \mathsf{I}$, which can be interpreted as the substitution $(x) \mapsto (x, x, \mathsf{i}x)$, gives rise to a ``restriction'' natural transformation of polynomial endofunctors (\cite{GambinoKoch:2013}),
\[
\rho^* : P_q \to P_\t\,,
\]
evaluation of which at $\t:\TT \to\T$ results in the following commutative naturality square.
\begin{equation}\label{diag:prod3}
\begin{tikzcd}
P_q\TT \ar[d,swap, "{P_q \t}"] \ar[r, "{\rho^{*}_{\TT}}"] & P_\t\TT \ar[d, "{P_\t \t}"] \\
P_q\T \ar[r,swap, "{\rho^{*}_{\T}}"] &  P_\t\T 
\end{tikzcd}
\end{equation}
Note that \eqref{diag:prod3} is a pullback in the extensional case; here we require it to be a \emph{weak} pullback by taking a section of the resulting ``gap map''.   Explicitly, \emph{weak pullback structure} $\mathsf{J}$ is a section of the resulting comparison map.
\begin{equation}\label{def:wps}
\begin{tikzcd}
P_q \TT  \ar[r] & \ar[l, bend right, dotted, "\mathsf{J}" description] P_q\T \times_{P_\t\T} P_\t\TT 
\end{tikzcd}
\end{equation}

To show that this models the standard elimination rule, namely
\[
\frac{x:A\vdash c(x) : C(\rho x)}{x, y:A, z : \Id_A(x, y)\vdash \mathsf{J}_c(x,y,z) : C(x,y,z)}
\]
take any object $\Gamma \in \EE$  and maps $(A, A, \Id_{A} \vdash C) : \Gamma \to P_q\T$ and $(A \vdash c) : \Gamma \to P_\t\TT$ with equal composites to $P_\t\T$, meaning that $A \vdash c : C(\rho x)$.  Composing the resulting map 
\[
(A\vdash c(x):C(\rho x)) : \Gamma \to P_q\T \times_{P_\t\T} P_\t\TT
\]
with $\mathsf{J} : P_q\T \times_{P_\t\T} P_\t\TT \to P_q \TT$ then indeed provides a term $$x:A, y:A, z:\Id_{A}(x,y)  \vdash \mathsf{J}_c(x,y,z) : C(x,y,z)\,.$$

The computation rule
\[
x: A \vdash \mathsf{J}_c(\rho x) = c(x) : C(\rho x)
\]
then says exactly that $\mathsf{J}$ is indeed a section of the comparison map \eqref{def:wps}.

\begin{proposition}[R.\ Garner]\label{prop:Id-typesfromML}  A natural model $\t : \TT \to \T$ satisfies the rules for intensional identity types just if the map $\t : \TT \to \T$ \emph{models} the same in the sense of Definition \ref{def:identitytype}: there are maps $(\mathsf{i}, \Id)$ making the diagram \eqref{diag:Idtypes} commute, together with a weak pullback structure $\mathsf{J}$ for the resulting comparison square \eqref{diag:prod3}.
\end{proposition}

The proposition clearly generalizes to arbitrary ML-algebras.

\section{Examples of ML-algebras}\label{sec:ExamplesML-Algebras}

\begin{example}\label{example:soc}
In an elementary topos $\EE$, the subobject classifier $\top : 1\to \Omega$ is an ML-algebra with the following structure maps.
\begin{equation*}\label{diag:MLalgO}
\begin{tikzcd}
	1 \ar[r,"!"] \ar[d,swap,"!"] \pbcorner & 1 \ar[d, "\top"]\\  
	1 \ar[r, swap,"\top"] & \Omega
 \end{tikzcd}\qquad\quad
 \begin{tikzcd}
	Q  \ar[r,"!"] \ar[d,swap,"{\top\cdot\top}"] \pbcorner &  1 \ar[d, "\top"]\\  
	\alg{P}_\top\Omega \ar[r,swap, "\Sigma"] & \Omega
\end{tikzcd} \qquad\quad
	 \begin{tikzcd}
	\alg{P}_{\top}1 \ar[d,swap,"\alg{P}_{\top}(\top)"] \ar[r,"!"] \pbcorner &  1 \ar[d, "\top"]\\  
	\alg{P}_{\top}\Omega \ar[r,swap, "\Pi"] & \Omega
 \end{tikzcd} 
 \end{equation*}
To specify the maps $\Sigma$ and $\Pi$, note that the polynomial functor $\alg{P}_\top$ is:
 \[
 \alg{P}_\top(X) = \Sigma_{p:\Omega}X^p = \widetilde{X}
 \]
where $\eta_X : X \mono \widetilde{X}$ is the topos-theoretic \emph{partial map classifier} \cite[Proposition A.2.4.7]{Johnstone:Elephant}. 
It follows easily from the specification of $\top\cdot\top : Q \to \alg{P}_\top\Omega$ in \eqref{diag:polyUMP2} that $\top\cdot\top$ is monic, and so there is indeed a (unique) classifying map $\Sigma : \alg{P}_{\top}\Omega \to \Omega$.   Similarly, since polynomial functors preserve monos, the map $\alg{P}_{\top}(\top) : \alg{P}_{\top}1 \to \alg{P}_{\top}\Omega$ (which is in fact the partial map classifier $\eta_\Omega : \Omega \mono \widetilde{\Omega}$) is also monic, and therefore also has a (unique) classifying map  $\Pi :\alg{P}_{\top}\Omega \to \Omega$.  See \cite[\S 6]{AGH} for details. 
\end{example}
\begin{example}\label{example:natmod}
As explained in Section \ref{sec:NaturalModels}, a natural model of type theory $\t: \TT\to \T$ over any index category $\CC$ is an ML-algebra in presheaves $\widehat\CC$ if it models the type formers $1, \Sigma, \Pi$.     For example, the groupoid model of \cite{Hofmann-Streicher} can be described as an ML-algebra in the category $\widehat{\mathsf{Gpd}}$ of presheaves on the category $\mathsf{Gpd}$ of groupoids. As is done in \emph{ibd.}, we can take $\T : \op{\mathsf{Gpd}} \to \Set$ to be the representable functor of the groupoid $\mathsf{gpd}$ of \emph{small} groupoids, where \emph{smallness} is defined as having cardinality strictly less than $\kappa$ for a suitably large cardinal number $\kappa$.  We may then furthermore take $\TT$ to be represented by $\dot{\mathsf{gpd}}$, the groupoid of small pointed groupoids, with $\t: \TT\to \T$ represented by the forgetful functor $\dot{\mathsf{gpd}} \to \mathsf{gpd}$, making the entire ML-algebra $\t: \TT\to \T$ representable.  Because the groupoid model has intensional identity types \cite{Hofmann-Streicher}, the corresponding natural model  $\t: \TT\to \T$ models them as well, as shown in \cite{awodey:NM}, and so the ML-algebra also has $\Id{}$-types in the sense of Definition \ref{def:EqualityandIdentityTypes}.

Observe that in this case, since the entire algebra $\t: \TT\to \T$ is representable, we don't really need the lccc $\widehat{\mathsf{Gpd}}$ for the notion of an ML-algebra to make sense; only that certain display maps are exponentiable in the index category $\mathsf{Gpd}$.  We shall not spell out the exact condition required here (roughly, being a $\Pi$-tribe in the sense of \cite{Joyal:CandT}), but only mention that the same is true for the syntactic model of Martin-L\"of type theory over the category $\mathsf{Ctx}$ of contexts and substitutions, which is \emph{a proiori} an ML-algebra in the category of presheaves $\widehat{\mathsf{Ctx}}$, but also has one in the non-lccc $\mathsf{Ctx}$ itself, if the type theory is assumed to have a universe $\mathsf{U}$ that is closed under $1, \Sigma, \Pi$.  Of course, we are using the fact that the Yoneda embedding preserves finite limits and exponentials, and therefore also ML-algebras
\end{example}
\begin{example}\label{example:strictification}
Generalizing the foregoing example, the strictification of an lccc $\EE$ (or just a category $\EE$ with a wfs) gives rise to a natural model with extensional equality (or intensional identity) types in $\widehat{\EE}$, as shown in \cite{awodey:NM}, and therefore to an ML-algebra over $\EE$ with the same.  Note that the map $\t: \TT\to \T$ is therefore tiny by Proposition \ref{prop:rep_is_tiny}, and therefore so is the entire ML-algebra structure, since representable maps are stable under pullback.  In the lccc case, the diagonal $\TT \to \TT\times_{\T}\TT$ interpreting the extensional equality types is therefore also tiny.
\end{example}
\begin{example}\label{example:HSuniverse}
In \cite{awodey:universes} the Hofmann-Streicher universe construction \cite{HS:1997} is analyzed using a realization-nerve adjunction between presheaves $\widehat\CC$ and $\Cat$, for any small index category $\CC$.  Briefly, the category of elements functor $\int_{\CC} : \widehat\CC \to \Cat$ is shown to have a right adjoint \emph{nerve functor} $\nu_\CC : \Cat \to \widehat\CC$, taking a small category $\A$ to the homset of functors $$\nu_\CC(\A)(c) = \Cat(\CC/c,\, \A)\,,$$ where $\CC/- : \CC \to \Cat$ is the \emph{covariant} slice category functor given by post-composition.
\begin{equation}\label{eq:nerve}\textstyle
\begin{tikzcd}
	 \widehat\CC \ar[rr, swap,"\int_\CC"] &&  \ar[ll, swap,bend right=20, "{ \nu_\CC}"] \Cat& \quad \int_\CC \dashv \nu_\CC\\  
	 \\
	\CC \ar[uu, hook, "\y"] \ar[rruu, swap,"{\CC/_{-}}"] &&&
 \end{tikzcd}
 \end{equation}
A Hofmann-Streicher universe $\VV \to \V$ in  $\widehat\CC$ is then shown to be the nerve of the classifier for discrete fibrations $\op\SSet\to\op\Set$ (suitably bounded in size by a cardinal $\kappa$, see \cite[\S 4]{awodey:universes}).    That is, we let
\begin{align}\label{eq:universedef}\textstyle
\VV_\kappa\, &=\, \nu\op{\dot{\Set}_\kappa}\\  
\V_\kappa\, &=\, \nu\op{\Set}_\kappa.\notag
 \end{align}
 
Now, the forgetful functor $U : \op{\dot{\Set}_\kappa}\to \op{\Set}_\kappa$ is a discrete fibration in $\Cat$, and therefore exponentiable. Moreover, one can show that if $\kappa$ is inaccessible, then $U : \op{\dot{\Set}_\kappa}\to \op{\Set}_\kappa$ will be an ML-algebra in $\Cat$, in the sense that there are structure maps $(*, 1, \sigma, \Sigma, \lambda, \Pi)$ making pullback squares 
as in Definition \ref{def:MLalg}, where the polynomial functor $\alg{P}_U : \Cat \to \Cat$ is used (even though $\Cat$ is not an lccc).

Finally, it can be shown that the right adjoint nerve functor $\nu_\CC : \Cat \to \widehat\CC$ commutes with polynomials, and therefore preserves ML-algebras.  It follows that any Hofmann-Streicher universe $\VV_\kappa \to \V_\kappa$  (for suitable $\kappa$) is an ML-algebra (with extensional equality) in $\widehat\CC$.  Indeed, the example of the subobject classifier $1 \to \Omega$ in a (presheaf) topos was an instance of this, since it is the nerve of the classifier $\mathbbm{1} \to \mathbbm{2}$ for sieves, which is itself an ML-algebra in $\Cat$, namely 
$$\mathbbm{1} = \SSet^{\mathsf{op}}_2 \too  \Set^{\mathsf{op}}_2 = \mathbbm{2}\,,$$  
so that in $\widehat\CC$ we have
\[
1 = \nu \mathbbm{1} = \VV_{2} \too  \V_{2} = \nu \mathbbm{2} = \Omega  \,.
\]
\end{example}

\section{A polynomial monad}\label{sec:Monad}
%

Let $\EE$ be an lccc, and recall from \cite[Theorem 2.17]{GambinoKoch:2013} the biequivalence 
\begin{equation}\label{eq:basicpolyequiv1}
\mathsf{Poly}^1_{\EE}\ \simeq\ \mathsf{PolyFun}^1_{\EE}
\end{equation}
between the 1-object bicategory $\mathsf{Poly}^1_{\EE}$ consisting of arrows $f:B \to A$ as 1-cells and cartesian squares
\begin{equation}\label{diag:cartsquare}
\begin{tikzcd}
B \ar[d,swap, "f"] \ar[r, "{h'}"] \pbcorner & D \ar[d, "g"] \\
A \ar[r,swap, "{h}"] &  C 
\end{tikzcd}
\end{equation}
as 2-cells, and the 1-object 2-category $\mathsf{PolyFun}^1_{\EE}$ consisting of polynomial endofunctors $P_f: \EE\to \EE$ and cartesian, strong natural transformations $\theta : P_f \Rightarrow P_g$ as 2-cells.  This is the 1-object special case of a more general biequivalence relating polynomials of the form $I \leftarrow B \to A \to J$ as 1-cells, between the objects $I, J$ of $\EE$ as 0-cells, and polynomial functors $\EE/_I \to \EE/_J$. (It also specializes at the level of 2-cells, where in the 1-object case we have cartesian strong natural transformations rather than all strong natural transformations).  

Simplifying the description in the 1-object case, we have an equivalence of 1-categories between the monoidal category $\EE^{\downarrow}_{\mathsf{cart}}$ consisting of cartesian squares in the arrow category and the strict monoidal subcategory of $\Cat(\EE, \EE)$ consisting of polynomial endofunctors and cartesian natural transformations, so that \eqref{eq:basicpolyequiv1} fits into the following diagram of 1-categories:
\begin{equation}\label{diagram:polyandpolyfun}
\begin{tikzcd}
& \mathsf{Poly}^1_{\EE} \ar[d,equals] \ar[r,"{\simeq}"] & \mathsf{PolyFun}^1_{\EE}  \ar[d,equals] &\\
\EE^{\downarrow}  &\EE^{\downarrow}_{\mathsf{cart}} \ar[l] \ar[r,"{\simeq}"] & \mathsf{PolyFun}(\EE, \EE)_{\mathsf{cart}}  \ar[r] & \mathsf{Fun}(\EE, \EE)
\end{tikzcd}
\end{equation}
Since cartesian natural transformations between polynomial endofunctors are always strong, we may omit mention of the strength in the polynomial case.

Under the equivalence in \eqref{diagram:polyandpolyfun}, a Martin-L\"of algebra in $\EE$ as in Definition~\ref{def:MLalg}, to wit,
\begin{equation}\label{diag:MLalg2}
\begin{tikzcd}
	1 \ar[r,"*"] \ar[d,swap,"!"] \pbcorner &  \TT \ar[d, "\t"]\\  
	1 \ar[r, swap,"\mathsf{1}"] & \T
 \end{tikzcd}\qquad \qquad 
 \begin{tikzcd}
	\TT_2  \ar[r,"\sigma"] \ar[d,swap,"\t\cdot\t"] \pbcorner &  \TT \ar[d, "\t"]\\  
	\T_2 \ar[r,swap, "\Sigma"] & \T
\end{tikzcd} \qquad \qquad 
	 \begin{tikzcd}
	\alg{P}_\t\TT \ar[r,"\lambda"] \ar[d,swap,"\alg{P}_\t(\t)"] \pbcorner &  \TT \ar[d, "\t"]\\  
	\alg{P}_\t\T \ar[r,swap, "\Pi"] & \T
 \end{tikzcd} 
 \end{equation}
gives rise to the following data in $\mathsf{PolyFun}(\EE, \EE)$:
\begin{enumerate}
\item a cartesian natural transformation $\tau :1_\EE \Rightarrow P_\t$\,,
\item a cartesian natural transformation $\sigma :P_\t\circ P_\t \Rightarrow P_\t$\,,
\item a cartesian arrow $\pi:P_\t(\t) \rightarrow \t$\, in $\EE^{\downarrow}$.
\end{enumerate}
Items (1) and (2) are of course familiar as the underlying structure of a monad; item (3) is a somewhat less familiar special case of an algebra structure map, with $P_\t$ lifted from $\EE$ to $\EE^{\downarrow}_{\mathsf{cart}}$ using the fact that polynomial functors preserve pullbacks.  The obvious question is now: are the respective \emph{equations} for a monad and an algebra satisfied by the natural transformations $\tau, \sigma, \pi$\,?  Unfortunately, the answer in general is \emph{no}, but as shown in \cite{Newstead:thesis, NA:2018}, there are reasons to investigate further.
 
Namely, following \emph{op.cit.}, the monad laws for $\tau$ and $\sigma$, when expressed as equations between natural transformations (in the left-hand column below) correspond to the following equations between objects in $\EE$ (in the right-hand column), as the components of the corresponding natural transformations.
\begin{align}\label{typeisos}\textstyle
 \sigma\circ P_\t\, \sigma\, &=\, \sigma\circ\sigma_{P_\t}	&&&	\sum\limits_{a:A}\sum\limits_{b:B(a)}C(a,b)\, &=\, \sum\limits_{(a,b):\sum\limits_{a:A}B(a)}C(a,b)\  \notag\\ 
  \sigma\circ P_\t\,\tau\, &=\, 1 		&&& 	\sum\limits_{a:A}1\, &=\, A  
  \\ 
  \sigma\circ \tau_{P_\t}\, &=\, 1		&&& 	\sum\limits_{x:1}A\, &=\, A  \notag
 \end{align}
And the algebra laws for $\pi$ correspond to the following equations between objects in $\EE$.
\begin{align}\label{typeisos2}
 \pi\circ P_\t\,\pi\, &=\,  \pi\circ\sigma 	&&&	\prod\limits_{a:A}\prod\limits_{b:B(a)}C(a,b)\, 
 									&=\, \prod\limits_{(a,b):\sum_{a:A}B(a)}C(a,b)  \notag\\ 
 \pi\circ \tau_{P_\t}\, &=\, 1 	&&& 	\prod\limits_{x:1}A\, &=\, A  
\end{align}

While the equations of objects in the right-hand column will not in general hold strictly, they do correspond to familiar \emph{type isomorphisms} that are provable from the usual rules of extensional dependent type theory with $1, \Sigma, \Pi, \mathsf{Eq}$, and, respectively, \emph{type equivalences} in the intensional case with $\mathsf{Id}$ in place of $\mathsf{Eq}$. 

In order to capture this apparently missing higher dimension of structure, in \cite{Newstead:thesis, NA:2018} the monoidal 1-category $\mathsf{PolyFun}(\EE, \EE)_{\mathsf{cart}}$ of endofunctors and cartesian natural transformations was promoted to a bicategory by defining certain 2-cells between the natural transformations, in such a way that $(\tau, \sigma)$ becomes a \emph{pseudo}-monad and $\pi$ a \emph{pseudo}-algebra.  The result was not entirely satisfactory, however, as the top dimension of structure was essentially codiscrete, so that the monad and algebra laws were not distinguished among all 2-cells between the particular parallel 1-cells in question.

Here we take a somewhat different approach, observing that the relevant type isomophisms resp.\ equivalences just mentioned are expressible in terms of the structure of an ML-algebra $\TT \to \T$ \emph{with $\Id{}$-types}, as defined in Section \ref{sec:Eq and Id}. The equations in \eqref{typeisos} and \eqref{typeisos2} 
 then all relate maps of the form $\alpha : X \to \T$, representing ``classified'' type families $A\to X$.  We can use the $\Id{}$-types of $\TT \to \T$ to define a suitable notion of equivalence for such families, up to which the desired equations can be shown to hold. 

\subsection{Classification}\label{sec:classify}

Let $\TT \to \T$ be an ML-algebra, and say that a map $a : A\to X$ in $\EE$ is \emph{classified} by $\TT \to \T$ if there exist maps $\alpha : X \to \T$ and $\dot{\alpha} : A \to \TT$ fitting into a pullback square as follows.
\begin{equation}\label{eq:classifiedfamilies}
\begin{tikzcd}
	 A \ar[d, swap, "a"] \pbcorner \ar[r, "\dot{\alpha}"] & \TT \ar[d, "\t"] \\  
	X \ar[r, swap, "\alpha"] & \T 
 \end{tikzcd}
 \end{equation}
 We may regard such a classification as a \emph{structure} on the map $a : A\to X$, of which there may be many (unless we assume that $\TT \to \T$ is univalent, see Remark \ref{remark:univalence}).  A morphism $f : a \to b$ of classified maps is a pullback square as on the left below making the diagram commute.
 \begin{equation}\label{eq:classifiedfamilies2}
\begin{tikzcd}
	A \ar[d,swap, "a"] \ar[r,swap, "\dot{f}"]  \ar[rr, bend left, "\dot{\alpha}"]  \pbcorner & B \ar[d,swap, "b"]   \ar[r,swap,  "\dot{\beta}"] \pbcorner & \TT \ar[d, "\t"]  \\  
	X \ar[r, "f"] \ar[rr,swap, bend right, "\alpha"] & Y  \ar[r, "{\beta}"] &  \T
 \end{tikzcd}
 \end{equation}
 Thus the category of classified maps is the slice category $\EE^{\downarrow}_{\mathsf{cart}}/_{\t}$  of the cartesian arrow category over $\t$\,.

The codomain functor
\[
\mathsf{cod} : \EE^{\downarrow}_{\mathsf{cart}}/_{\t} \too \EE/_\T
\]
is an equivalence of categories, with a right adjoint $\t^{\sharp} : \EE/_\T \to \EE^{\downarrow}_{\mathsf{cart}}/_{\t}$ choosing, for each object $\alpha : X \to \T$ in $\EE/_\T$\,, a pullback span $$X \stackrel{\ a}{\leftarrow} A \stackrel{\dot{\alpha}\ }{\to} \TT$$ as in \eqref{eq:classifiedfamilies}, which determines the strictly functorial action on arrows.  As an adjoint, the functor $\t^{\sharp}$ is of course uniquely determined up to natural isomorphism.  Summarizing, we have an adjoint equivalence of 1-categories:
 \begin{equation}\label{eq:slice adjunction}
\begin{tikzcd}
 {\EE^{\downarrow}_{\mathsf{cart}}/_{\t}}  \ar[rr, bend right, swap, "{\mathsf{cod} }"] \ar[rr, phantom, description, "{\simeq}"] && \ar[ll, bend right, swap, "{\t^{\sharp}}"] \EE/_\T & \quad \mathsf{cod} \dashv \t^{\sharp}
 \end{tikzcd}
 \end{equation}

Note that we also have a natural model of type theory on the index category $\EE$, given by (the Yoneda embedding applied to) $\t: \TT \to \T$, with the category of elements of the functor $\y\T$ of types in context being $\int_{\EE}\!\T \simeq \EE/_\T$ and the category of elements functor of the arrow $\t: \TT \to \T$ being the post-composition functor $\t_{!} : \EE/_{\TT} \to \EE/_{\T}$.
 \begin{equation}\label{eq:sliceadjunctionelements}
\begin{tikzcd}
\int_{\EE}\!\TT \ar[d,swap, "{\int_{\EE}\!\t}"] \ar[r, equals] & \EE/_{\TT} \ar[d, "{\t_{!}}"] \ar[rd] & \\
 \int_{\EE}\!\T \ar[r,equals] &  \EE/_\T \ar[r] & \EE
 \end{tikzcd}
 \end{equation}
 In terms of this natural model, the functor $\t^{\sharp} : \EE/_\T \to \EE^{\downarrow}_{\mathsf{cart}}/_{\t}$ taking $\alpha : X \to \T$ to $\alpha^*\t = a: A \to X$, equipped with a cartesian square $(\alpha, \dot{\alpha}) : \alpha^*\t \to \t$ as in \eqref{eq:classifiedfamilies}, is the usual ``comprehension'' or context extension operation determining the right adjoint to $\int_{\EE}\!\t$.
 
Finally, let $\EE_\t \hook  \EE^{\downarrow}_{\mathsf{cart}}$ be the essential image of the forgetful functor 
\[
U_\t :  \EE^{\downarrow}_{\mathsf{cart}}/_{\t} \twoheadrightarrow \EE_\t \hook \EE^{\downarrow}_{\mathsf{cart}}
\]
that takes a classified map \eqref{eq:classifiedfamilies} to the map $a : A \to X$, forgetting its classifying square.  
The category $\EE_\t$ thus consists of those arrows $a:A \to X$ that are ``classifiable'' by $\t$, with cartesian squares between them as morphisms. As a subfibration of the codomain fibration 
\[
\EE_\t \hook \EE^\downarrow \stackrel{\mathsf{cod}}{\epi} \EE
\]
 the fibration $\EE_\t \epi \EE$ is not discrete, unlike $\EE^{\downarrow}_{\mathsf{cart}}/_{\t} \twoheadrightarrow \EE/_\T$, and no longer comes with a splitting, as was given by the choice $\u^\sharp$ of classified maps $a : A\to X$ for each $\alpha : X \to \T$.  These classifiable maps endow $\EE$ with the structure of a \emph{category with display maps} in the sense of \cite{Taylor:PFM}.  

\begin{remark}\label{remark:classifyingmaps}
In virtue of the ML-algebra structure on $\TT\to\T$, the classifiable maps $A\to X$ are still closed under the type formers $1, \Sigma, \Pi$, when  regarding objects $X \in \EE$ as contexts and classifiable maps $A\to X$ as families of types in context~$X$, as usual for (``incoherent'') models of dependent type theory in clans, categories with display maps, or similar (see Remark \ref{remark:clans}).  However, in what follows, we prefer to use the split, discrete fibration $\EE/_\T \to \EE$ as a model of type theory, because it comes with the strictly coherent natural model structure of \eqref{eq:sliceadjunctionelements}.  We can then regard the objects $\alpha : X \to \T$ as \emph{type families in context $X$}, and refer to them interchangeably also in terms of their (canonically chosen) display maps $\alpha^*\t := a : A \to X$.  See \cite[\S 1--3]{AGH} for a systematic development of this solution to the familiar ``coherence problem'' \cite{Hofmann:1994} for modeling dependent type theory, which derives originally from Voevodsky \cite{KL:VV}.  
\end{remark}

\subsection{Equivalence of types}\label{sec:equivalence}

Now suppose that the ML-algebra $\TT \to \T$ in $\EE$ also has the structure required for (extensional $\mathsf{Eq}$-types or, more generally, intensional) $\Id{}$-types, in the sense of Definition \ref{def:EqualityandIdentityTypes}.  By Proposition \ref{prop:Id-typesfromML}, this endows all type families $\alpha : X \to \T$ with identity types, which we may write interchangeably as $\Id_{A}(a, b)$ and $(a =_A b)$.   We can then use these to define an \emph{equivalence} of types $A\simeq B$ (over the base $X$, which we suppress in what follows) as in \cite[\S 4.3]{HoTTbook}, namely:
\begin{definition}
 A map $e : A\to B$ is a (type) \emph{equivalence} if it has both right and left inverses, in the sense that there are maps $f, g : B\tto A$ such that 
 \[
f\circ e =_{A^A} 1_A \qquad\text{and}\qquad  e\circ g =_{B^B} 1_B  \,.
 \]
 Formally, given $e : A\to B$ we define 
 \[
 \mathsf{isEquiv}(e)\ :\equiv\ \Sigma_{f, g: B\to A}\,  \Pi_{x:A}\Id_{A}(fe(x) , x) \times \Pi_{y:B}\Id_{B}(eg(y), y) \,.
 \]
Then for $A, B: \T$, we define the notation $A \simeq B$ by
 \begin{equation}\label{eq:typeequivalence}
A \simeq B\, =\, \mathsf{Equiv}(A,B)\ :\equiv\  \Sigma_{e : A\to B}\, \mathsf{isEquiv}(e)\,.
 \end{equation}
\end{definition}
We can construct (the interpretation of) the universal equivalence of type families $A, B : \T \vdash A \simeq B$ as a map $\mathsf{Equiv} \to \T \times \T$ as follows (cf.\ \cite{KL:VV}).

First, pull $\TT \to \T$ back along the two different projections $\T \times \T \tto \T$ to obtain two different objects $\TT_1, \TT_2$ over $\T \times \T$, and then take their exponential $[\TT_1, \TT_2]$ in the slice category over ${\T \times \T}$, which interprets the type family $A, B : \T \vdash A \to B$.
\begin{equation*}\label{diag:univalence1}
\begin{tikzcd} 
	\TT \ar[d] & \TT_1 \ar[l] \ar[rd] \pbcornerright &\ar[d] [\TT_1, \TT_2] & \TT_2 \ar[ld] \ar[r] \pbcorner & \TT \ar[d] \\  
	\T && \ar[ll]  \T \times \T \ar[rr] && \T
	 \end{tikzcd}
 \end{equation*}
 Indeed, after pulling $\TT_1, \TT_2$ back to $[\TT_1, \TT_2]$ there is a (universal) arrow $\varepsilon : \TT'_1 \to \TT'_2$ over $[\TT_1, \TT_2]$, as indicated below.
 \begin{equation}\label{diag:univalence11}
\begin{tikzcd} 
& \TT'_1  \ar[d] \arrow[dr,phantom,"\lrcorner" very near start, shift right=1.5ex] \ar[rd] \ar[rr,"\varepsilon"] 
	&& \TT'_2 \ar[ld]  \arrow[dl,phantom,"\llcorner" very near start, shift left=1.5ex] \ar[d] & \\  
	\TT \ar[d] & \TT_1 \ar[l] \ar[rd] \pbcornerright &\ar[d] [\TT_1, \TT_2]& \TT_2 \ar[ld] \ar[r] \pbcorner & \TT \ar[d] \\  
	\T && \ar[ll]  \T \times \T \ar[rr] && \T
	 \end{tikzcd}
 \end{equation}
To build the equivalence classifier, let $\mathsf{Equiv} \to \T \times \T$ be the composite of the canonical projection $\mathsf{isEquiv}(\varepsilon) \to [\TT_1, \TT_2]$ with the map $[\TT_1, \TT_2] \to  \T \times \T$, as indicated below.
\begin{equation}\label{diag:univalence2}
\begin{tikzcd} 
{\mathsf{isEquiv}(\varepsilon)}  \ar[d, equals] \ar[r] & {[\TT_1, \TT_2]}  \ar[d]\\
\mathsf{Equiv} \ar[r] &  \T \times \T
 \end{tikzcd}
 \end{equation}
Given $\alpha, \beta : X \to \T$ classifying types $A\to X$ and $B \to X$, factorizations of the map $(\alpha, \beta) : X \to \T \times \T $ through $\mathsf{Equiv} \to \T \times \T$ can then be seen using \eqref{diag:univalence11}  to correspond to equivalences $A\simeq B$ over $X$, as indicated in the following, in which the indicated map $\mathsf{E}_1\simeq\mathsf{E}_2$ is the pullback of  $\varepsilon : \TT'_1\to \TT'_2$ along $\mathsf{Equiv}\to {[\TT_1, \TT_2]}$.
\begin{equation}\label{diag:univalence3}
\begin{tikzcd} 
	&& \mathsf{E}_1 \ar[rd]  \ar[r,"{\simeq}"]  & \mathsf{E}_2 \ar[d] \\  
A \ar[rd] \ar[r, "{\simeq}"] \ar[rru, dotted] & B \ar[d] \ar[rru, crossing over, dotted] && \mathsf{Equiv} \ar[d] \\  
& X \ar[rru, dotted] \ar[rr, swap, "{(\alpha, \beta)}"] && \T \times \T
	 \end{tikzcd}
 \end{equation}
The correspondence is natural in $X$, because it is mediated by pulling back a universal instance.  

The structure $\mathsf{Equiv} \tto \T $ is an internal graph in $\EE$ that, moreover, comes with reflexivity $\T \to \mathsf{Equiv}$, symmetry $\mathsf{Equiv} \to \mathsf{Equiv}$ and transitivity $\mathsf{Equiv}\times_\T \mathsf{Equiv} \to \mathsf{Equiv}$ maps over $\T$, in virtue of the identity types, and the resulting nature of type equivalences; indeed, it could be shown to be an ``internal weak $\infty$-groupoid'' in the sense of \cite{Lumsdaine:groupoids}, but we need not spell that out further.  Its purpose here is merely to endow the category $\EE/_\T$ of classified types $\alpha : X \to \T$ with the equivalence relation $A\simeq B$ of \eqref{eq:typeequivalence} by ``homming in'' to it, as explained in the introduction.  

The following lemma shows that, for a map $e : A \to B$ (over a base $X$), the property of being an equivalence of types is independent of the choice of classifying maps of $A$ and $B$.
\begin{lemma}\label{lemma:equivinvariance}
Suppose given $A\to X$ classified by $\alpha: X \to \T$ and $B \to X$ classified by $\beta: X \to \T$, and a map $e : A \to B$ over $X$ that is an equivalence in virtue of a lift $\tilde{e} : X \to  \mathsf{Equiv}$ of $(\alpha, \beta) : X \to \T\times \T$.   Then upon choosing different classifying maps $\alpha', \beta' : X \to \T$, the map $e : A \to B$ over $X$ is still an equivalence, now in virtue of a lift ${\tilde{e}}' : X \to \mathsf{Equiv}$ of $(\alpha', \beta') : X \to \T\times \T$.
\end{lemma}

\begin{proof}
The map $1_A : A \to A$ over $X$ is an equivalence with respect to $\alpha: X \to \T$ and $\alpha': X \to \T$, in virtue of a lift $\ell(\alpha, \alpha') : X \to\mathsf{Equiv}$ of $(\alpha, \alpha') : X \to \T\times \T$, by the universal property of $\mathsf{Equiv} \to \T\times \T$, and the same is true, \emph{mutatis mutandis}, for  $1_B : B \to B$.  Now the structure $\mathsf{Equiv} \to \T\times \T$ can be shown to have a composition $(e\cdot e')$ operation in the expected way, using the rules of type theory \cite[Theorem 4.7.1]{HoTTbook}, which we know are sound when interpreted by the operations of the ML-algebra.  So we may define the desired lift ${\tilde{e}}' : X \to \mathsf{Equiv}$ to be e.g. $${\tilde{e}}' = \ell(\beta, \beta') \cdot ({\tilde{e}}\cdot\ell(\alpha', \alpha))\,.$$
\end{proof}
Lemma \ref{lemma:equivinvariance} tells us that we could work in the category $\EE_\t$ of classifiable types ``up to equivalence'', but we prefer to continue assuming chosen classifying maps when convenient, working instead in its ``presentation'' $\EE/_\T$ as explained in Remark \ref{remark:classifyingmaps}.

\subsection{The category of polynomials}\label{sec:2cells}

Recall the category of polynomials, 
\[
\EE^{\downarrow}_{\mathsf{cart}}\ \simeq\ \mathsf{PolyFun}(\EE, \EE)_{\mathsf{cart}}
\]
as in \eqref{diagram:polyandpolyfun}, with objects $f, g$ and arrows $(h, h') : f \to g$ as in \eqref{diag:cartsquare}, corresponding to polynomial endofunctors $\alg{P}_f, \alg{P}_g : \EE \to \EE$ and cartesian natural transformations $\vartheta : \alg{P}_f \Rightarrow \alg{P}_g$.

We can now restrict $\EE^{\downarrow}_{\mathsf{cart}}$ to the (non-full) subcategory $U : \EE/_\T \to \EE^{\downarrow}_{\mathsf{cart}}$ of ``classified types'' determined by composing,
\begin{equation}\label{eq:faithfulU}
U : \EE/_\T\ \stackrel{\t^{\sharp}}{\simeq}\ \EE^{\downarrow}_{\mathsf{cart}} /_\t \stackrel{U_\t}{\too} \EE^{\downarrow}_{\mathsf{cart}}\,,
\end{equation}
by choosing as 0-cells only the canonical  pullbacks $a = \alpha^{*}\t$ and as 1-cells only those maps that commute with the associated classifying maps:
\begin{equation}\label{diag:classified1cells}
\begin{tikzcd}
A \ar[d,swap, "a"] \ar[r, "{h'}"] \pbcorner & B \ar[d, "b"] \\
X \ar[r,swap,swap, "{h}"] \ar[rd,swap, "\alpha"] &  Y \ar[d, "\beta"]\\
& \T
\end{tikzcd}
\end{equation}
This determines a correspondingly restricted subcategory of ``$\T$-classified'' polynomial endofunctors and cartesian natural transformations, derived from \eqref{diagram:polyandpolyfun}:
\begin{equation}\label{diagram:polyandpolyfun2}
\begin{tikzcd}
\EE/_{\T} \ar[d,swap,"U"] \ar[r,"{\simeq}"] & \mathsf{PolyFun}(\EE, \EE)_{\mathsf{cart}}^\T  \ar[d] \\
\EE^{\downarrow}_{\mathsf{cart}} \ar[r,"{\simeq}"] & \mathsf{PolyFun}(\EE, \EE)_{\mathsf{cart}}
\end{tikzcd}
\end{equation}

The category $\mathsf{PolyFun}(\EE, \EE)_{\mathsf{cart}}^\T$ consists of polynomial endofunctors $P_a : \EE\to \EE$ whose signature $a: A \to X$ is a classified type family $\alpha : X \to \T$.

Since the ``type relation'' $\mathsf{Equiv} \to \T\times \T$ classifies equivalences of such families of types \eqref{diag:univalence3}, it therefore also endows 
the classified polynomial endofunctors $\mathsf{PolyFun}(\EE, \EE)_{\mathsf{cart}}^\T$ with such a structure. 

\begin{proposition}\label{prop:monad}
For an ML-algebra $\t : \TT \to \T$ with $\Id{}$-types, the classified polynomial endofunctors $\mathsf{PolyFun}(\EE, \EE)_{\mathsf{cart}}^\T$ of \eqref{diagram:polyandpolyfun2} satisfy the polynomial monad and algebra equations listed in \eqref{typeisos} and \eqref{typeisos2} up to type equivalence $A\simeq B$, as defined in \eqref{eq:typeequivalence}.\qed
\end{proposition}

Using the interpretation from Section \ref{sec:NaturalModels}, the proof now simply amounts to verifying the equivalences \eqref{typeisos} and \eqref{typeisos2}  in type theory (as can be found in \cite[Chapter 2]{HoTTbook}).  We leave the detailed formal development to the formalization project \cite{HoTTLean}.
  
\begin{remark}[Univalence]\label{remark:univalence}

If the ML-algebra $\T$ is itself a universe type $\T : \T_1$ in a larger algebra $\TT_1 \to \T_1$ also with $\Id{}$-types, then there is the type family $A, B: \T \vdash \Id_{\T}(A,B)$, which is interpreted as a map $\Id_{\T}\to \T \times \T$.  Since for every $A:\T$ the map $1_A : A \to A$ is an equivalence, by $\Id_{\T}$-elim we obtain a distinguished map  $\Id_{\T}(A, B) \to \mathsf{Equiv}(A, B)$ over $\T \times \T$.  The \emph{univalence axiom} is the statement that this map is itself an equivalence:
\begin{equation}\label{diag:univalence4}
\begin{tikzcd} 
\Id_{\T} \ar[d] \ar[r, "{\simeq}"] & \mathsf{Equiv} \ar[d] \\  
 \T \times \T  \ar[r, equals] & \T \times \T
	 \end{tikzcd}
 \end{equation}
 Pulling the equivalence back along any $(A, B) : X \to \T \times \T$ results in the more familiar formulation:
 \[\tag{UA}
 \ (A =_\T B) \, \simeq\, (A \simeq B)
 \]
The interpretation of univalence thus requires a universe $\T$ with $1$, $\Sigma$, $\Pi$, $\Id{}$ (in order to define $A\simeq B$), inside a larger universe $\T:\T_1$ also with $1, \Sigma, \Pi, \Id{}$ (to have the family $A =_\T B$ and state the central equivalence in UA). The operations on the larger universe are usually required to restrict to the corresponding ones on the smaller universe.   

Of course, the same situation could also be described as a single univalent universe $\T$ with $1, \Sigma, \Pi, \Id{}$ in a (higher) category with $\Id{}$, etc.  When univalence is available, the monad and algebra laws of \eqref{typeisos} and \eqref{typeisos2}  become (internal) identities, rather than type equivalences.  This is the approach taken in \cite{AberleSpivak:2024}, where a distributive law replaces our algebra structure.
\end{remark}

\begin{example}
It is instructive to work out the slice category $\EE/_\T$ and the associated endofunctor category $\mathsf{PolyFun}(\EE, \EE)_{\mathsf{cart}}^\T$ for a familiar example of an ML-algebra $\t:\TT\to\T$. Thus consider Example \ref{example:HSuniverse} from Section \ref{sec:ExamplesML-Algebras} for the case $\kappa = \omega$ in the ambient lccc $\EE = \Set$.  

As the base object $\T$ we can take the set $\N$ of natural numbers, and as $\t:\TT\to\T$ the set of finite cardinal numbers $\NN = \{ [n]\,|\, n\in\N\}$, indexed by their cardinality via $\n : \NN \to \N$.   The polynomial endofunctor $$\alg{P}_\n : \Set \to \Set$$ is therefore:
\begin{align*}
\alg{P}_\n(X) &= \Sigma_{n:\N}X^{n} = 1 + X + X^2 + \dots \\
& = X^*
\end{align*}
where $X^*$ is the set of finite lists $\langle x_1, \dots, x_n\rangle$ of elements $x_i \in X$.   Of course, this is well known to be a (polynomial) monad, namely the free monoid monad \cite[Example 1.9]{GambinoKoch:2013}.  We leave it to the reader to work out the ML-algebra structure:
\begin{equation*}\label{diag:MLalgN}
\begin{tikzcd}
	1 \ar[r,"\{0\}"] \ar[d] \pbcorner & \NN \ar[d, "\n"]\\  
	1 \ar[r, swap,"1"] & \N
 \end{tikzcd}\qquad\qquad
 \begin{tikzcd}
	Q  \ar[r,"\sigma"] \ar[d,swap,"{\n\cdot\n}"] \pbcorner &  \NN \ar[d, "\n"]\\  
	\N^* \ar[r,swap, "\Sigma"] & \N
\end{tikzcd} \qquad\qquad
	 \begin{tikzcd}
	\NN^* \ar[d,swap,"{\n^*}"] \ar[r,"\lambda"] \pbcorner &  \NN \ar[d, "\n"]\\  
	\N^* \ar[r,swap, "\Pi"] & \N
 \end{tikzcd} 
 \end{equation*}
Because $\n : \NN \to \N$ is injective, the diagonal on the left below is an iso, and so the following is already a pullback when we take $\mathsf{Eq}$ to be the projection and $\refl$ the identity.
 \begin{equation*}\label{diag:EqtypesN}
\begin{tikzcd}
	\NN \ar[r,"\refl"] \ar[d,swap,"\delta"] &  \NN \ar[d, "\n"]\\  
	\NN \times_\N \NN \ar[r, swap,"\mathsf{Eq}"] & \N
 \end{tikzcd}\qquad \qquad 
 \end{equation*}
This algebra therefore has \emph{extensional} $\mathsf{Eq}$-types. 
It follows that in this case, the type relation $\mathsf{Equiv} \tto \N$ constructed in \eqref{diag:univalence2} is actually a 1-groupoid, namely the skeleton of the groupoid of finite sets and bijections, the coproduct of the finite symmetric groups $\coprod_n\Sigma_n\to\N$. 

The slice category $\Set/_\N$ has as objects, maps $\alpha: X\to \N$ classifying canonical families  $a :A \to X$ with finite fibers and a chosen pullback square:
\begin{equation}\label{eq:classifiedfamiliesN}
\begin{tikzcd}
	 A \ar[d, swap, "a"] \pbcorner \ar[r, "\dot{\alpha}"] & \NN \ar[d, "\n"] \\  
	X \ar[r, swap, "\alpha"] & \N 
 \end{tikzcd}
 \end{equation}
 The 1-cells are maps $f : X\to Y$ commuting over $\N$.  Two such families $A, B \to X$ are equivalent just if they are fiberwise isomorphic over $X$.  We have an equivalence of categories,
\[
\Set/_\N\ \simeq\ \mathsf{PolyFun}(\Set, \Set)_{\mathsf{cart}}^\mathsf{fin}
\]
with the 1-category of polynomial endofunctors with finite fibers and cartesian natural transformations.  
\end{example}
 
\subsection*{Acknowledgements.}

\quad Thanks to Mathieu Anel, Corinthia Aberl\'e, Marcelo Fiore, Nicola Gambino, Richard Garner, and Clive Newstead for ideas and advice, as well as to Yiming Xu and an anonymous referee for comments on an earlier version.  This material is based upon work supported by the Air Force Office of Scientific Research under awards FA9550-21-1-0009 and FA9550-20-1-0305.

\bibliographystyle{alpha}
\newcommand{\etalchar}[1]{$^{#1}$}

\end{document}